\documentclass[11pt,reqno]{amsart}

\usepackage[foot]{amsaddr}
\usepackage{amsfonts}
\usepackage{hyperref}
\usepackage{mathrsfs}
\usepackage{amsthm}
\usepackage{amsmath}
\usepackage{amssymb}
\usepackage{graphicx}
\usepackage{tikz-cd}
\usepackage{comment}

\theoremstyle{plain}
\newtheorem{Th}{Theorem}
\newtheorem{Le}[Th]{Lemma}

\newtheorem{Pro}[Th]{Proposition} 
\newtheorem{Cor}[Th]{Corollary}
\newtheorem{Rem}[Th]{Remark}

\newcommand{\thistheoremname}{\hat}
\newtheorem*{genericthm*}{\thistheoremname}
\newenvironment{namedthm*}[1]
  {\renewcommand{\thistheoremname}{#1}%
   \begin{genericthm*}}
  {\end{genericthm*}}

\def \Hnc{\mathbb{H}_{\mathbb{C}}^n}
\def\dHnc{\partial\mathbb{H}_{\mathbb{C}}^n}
\def \Hc{\mathbb{H}_{\mathbb{C}}^2}
\def\dHc{\partial\mathbb{H}_{\mathbb{C}}^2}
\newcommand{\Sym}{\text{Sym}(5)}

\newcommand{\la}{\langle } 
\newcommand{\ra}{\rangle}
\newcommand{\A}{\mathbb{A}}

\begin{document}

\title[]{New bounds for the Simplicial volume of complex hyperbolic surfaces}
\author[]{Hester Pieters}
\address{Weizmann Institute of Science, Rehovot, Israel}
\email{hester.pieters@weizmann.ac.il}
\thanks{This research was supported by Swiss National Science Foundation grant number PP00P2-128309/1.}
\date{}
\maketitle

\begin{abstract}
We give estimates of the Gromov norm of the top dimensional class in $H_c^4(\mathrm{Isom}(\Hc);\mathbb{R})$. As a consequence, we obtain an explicit upper bound for the simplicial volume of closed oriented manifolds that are locally isometric to $\Hc$. 
\end{abstract}

\section{Introduction}
\label{intro}

Simplicial volume was introduced by Gromov in \cite{Gr}. It gives a topological measure of the complexity of a manifold. Until now its exact value has only been computed for hyperbolic manifolds \cite{Gr}, \cite{Th1} and for closed manifolds covered by $\mathbb{H}_{\mathbb{R}}^2\times\mathbb{H}_{\mathbb{R}}^2$ \cite{Bu2}. Except for these results no explicit upper bounds for the simplicial volume are known. There are some more nonvanishing results. For example, the simplicial volume of oriented closed connected locally symmetric spaces of non-compact type is nonzero \cite{LS}. Also, in the case of negative curvature the simplicial volume is bounded from below by the Riemannian volume and therefore nonzero \cite{Gr}, \cite{Th1}. However, in general there is also no explicit lower bound known. 
\\\\
Let $\Hc$ be the complex hyperbolic plane with holomorphic curvature $-1$ and let $[c_\Phi]\in H^2_{c,b}(\mathrm{PU}(2,1);\mathbb{R})$ be the K\"ahler class. Our main result is 

\begin{Th}
\label{kahler}
\[
\frac{2}{9}\pi^2\leq\|[c_\Phi\cup c_\Phi]\|_\infty\leq\pi^2
\]
\end{Th} 

As $c_\Phi\cup c_\Phi$ is proportional to the image under the van Est isomorphism of the volume form in $\Omega^4(\Hc;\mathbb{R})^{\mathrm{PU}(2,1)}$ applying the proportionality principle gives explicit bounds for the simplicial volume $\| M\|$ of a closed complex hyperbolic surface $M$:
\begin{Th}
\label{simplicialvolume}
Let $M$ be a closed oriented manifold which is locally isometric to $\Hc$. Then
\[
\frac{2}{\pi^2}\mathrm{Vol}(M)\leq \|M\|\leq \frac{9}{\pi^2}\mathrm{Vol}(M).
\]
\end{Th}

From the Hirzebruch proportionality principle relating the volume of $M$ to its Euler characteristic (\cite{Hi}) we get bounds for the simplicial volume of $M$ in terms of its Euler characteristic:
\begin{Cor}
\label{ineq_chi}
Let $M$ be a closed oriented manifold which is locally isometric to $\Hc$. Then
\[
\frac{16}{3}\chi(M)\leq \|M\|\leq 24\chi(M). 
\]
\end{Cor}

The following Milnor-Wood inequality immediately follows:
\begin{Cor}
\label{MW}
Let $\xi$ be a flat $\mathrm{GL}^+(4,\mathbb{R})$-bundle over a closed complex hyperbolic surface $M$. Then
\[
|\chi(\xi)|\leq \frac{3}{2}\chi(M).
\]
\end{Cor}

This paper is structured as follows: In Section 2 we recall the necessary background on complex hyperbolic geometry, continuous (bounded) cohomology and simplicial volume. 
In Section 3 we show in detail how Theorem \ref{simplicialvolume}, Corollary \ref{ineq_chi} and Corollary \ref{MW} follow from Theorem \ref{kahler}. Finally, in Section 4 we give the proof of Theorem \ref{kahler}. 

\subsection*{Acknowledgement}
The author would like to thank her doctoral advisor Michelle Bucher for introducing her to the subject and for many helpful discussions and advice.

\section{Preliminaries}
\label{Preliminaries}
\subsection{Complex hyperbolic plane}
\label{Complex hyperbolic plane}
Let $\mathbb{C}^{2,1}$ be the complex vector space $\mathbb{C}^3$ equipped with the Hermitian form 
\[
\la\mathbf{z},\mathbf{w}\ra:=z_1\overline{w_1}+z_2\overline{w_2}-z_3\overline{w_3}.
\]
Since $\la \mathbf{z},\mathbf{z}\ra$ is real for all $\mathbf{z}\in\mathbb{C}^{2,1}$ we can define the following subsets of $\mathbb{C}^{2,1}$:
\begin{align*}
&V_-:=\left\{\mathbf{z}\in\mathbb{C}^{2,1}\middle| \la \mathbf{z},\mathbf{z}\ra<0\right\},\\
&V_0:=\left\{\mathbf{z}\in\mathbb{C}^{2,1}\setminus\{0\}\middle| \la \mathbf{z},\mathbf{z}\ra=0\right\},\\
&V_+:=\left\{\mathbf{z}\in\mathbb{C}^{2,1}\middle| \la \mathbf{z},\mathbf{z}\ra>0\right\}.
\end{align*}
Denote by $\mathbb{P}$ the canonical projection of $\mathbb{C}^{2,1}\setminus\{0\}$ onto $\mathbb{C}P^2$. The \textit{projective model} of complex hyperbolic $n$-space $\Hc$ is then defined to be $\mathbb{P}(V_-)$ and the boundary, $\dHc$, is defined to be $\mathbb{P}(V_0)$. The metric on $\Hc$ is defined by 
\[
\cosh^2\left(\frac{1}{2}d({z},{w})\right)=\frac{\la \bf{z},\bf{w}\ra\la w,z\ra}{\la \bf{z},z\ra\la \bf{w},w\ra},
\]
with $d(\cdot,\cdot)$ the distance function. With this scaling $\Hc$ has holomorphic sectional curvature $-1$, with real sectional curvature pinched between $-1$ and $-1/4$. The holomorphic isometry group of $\Hc$ is $\mathrm{PU}(2,1)$ while its full isometry group $\widehat{\mathrm{PU}(2,1)}$ is generated by $\mathrm{PU}(2,1)$ and complex conjugation.

Let $\eta\in\mathrm{U}(1)$ and let $L$ be a complex line, i.e. the intersection of a complex line in $\mathbb{C}P^2$ with $\Hc$. The \textit{complex reflection in $L$ with reflection factor $\eta$} is the map $\rho_L^\eta:\mathbb{C}^{2,1}\to\mathbb{C}^{2,1}$ defined by 
\[
\mathbf{z}\mapsto \mathbf{z}+(\eta-1)\frac{\la\mathbf{z},\mathbf{c}\ra}{\la\mathbf{c},\mathbf{c}\ra}\mathbf{c},
\]
for $\mathbf{z}\in\mathbb{C}^{2,1}$ and where $\mathbf{c}$ is a polar vector of $L$, i.e. $L=\mathbb{P}(\mathbf{c}^\perp)\cap\Hc$. If $\eta=-1$ we simply call this map the \textit{complex reflection in $L$}. 
\\\\
The \textit{Cartan angular invariant} $\A$ of a triple $(p_1,p_2,p_3)\in(\dHc)^3$ is by definition
\[
\A(p_1,p_2,p_3)=\arg(-\la\mathbf{p_1},\mathbf{p_2}\ra\la \mathbf{p_2},\mathbf{p_3}\ra\la\mathbf{p_3},\mathbf{p_1}\ra).
\]
It characterizes triples in $\dHc$ up to the action by $\mathrm{PU}(2,1)$ (\cite[Theorem 7.1.1]{Gol}). 
In the following Lemma we list some useful and elementary properties of $\A$.
\begin{Le}{\cite[Section 7.1]{Gol}}
The Cartan angular invariant $\A$ has the following properties
\begin{enumerate}
\item $\A$ is alternating: If $\sigma\in \mathrm{Sym}(3)$ then 
\[
\A(p_{\sigma(1)},p_{\sigma(2)},p_{\sigma(3)})=\mathrm{sign}(\sigma)\A(p_1,p_2,p_3).
\]
\item If $g\in \mathrm{PU}(n,1)$ is a holomorphic automorphism then 
\[
\A(gp_1,gp_2,gp_3)=\A(p_1,p_2,p_3),
\]
and if $g$ is an anti-holomorphic automorphism then 
\[
\A(gp_1,gp_2,gp_3)=-\A(p_1,p_2,p_3).
\]
\end{enumerate}
\end{Le}

\subsection{The volume form in $H_{c,b}^*(\mathrm{PU}(2,1),\mathbb{R})$.}
\subsubsection{Continuous (bounded) cohomology}
Let $G$ be a topological group and $A$ a Banach G-module. We will consider abstract groups as topological groups with respect to the discrete topology. Let
\[
C_{c,b}^p(G;A):=\{f:G^{p+1}\to A\mid f \text{ continuous and bounded}\},
\]
with the $G$-action given by
\[
(g\cdot f)(g_0,\dots,g_p):=g\cdot(f(g^{-1}g_0,\dots, g^{-1}g_p)).
\]
We denote by $C_{c,b}^p(G;A)^G$ the space of $G$-invariant functions. Let 
\[
\delta:C_{c,b}(G^{p+1};A)^G\to C_{c,b}(G^{p+2};A)^G
\]
be the standard homogeneous coboundary operator, i.e. for a cochain $\alpha\in C_{c,b}(G^{p+1};A)^G$ and $g_0,\dots,g_{p+1}\in G$
\begin{eqnarray*}
\delta\alpha(g_0,\dots,g_{p+1}):= \sum_{i=0}^{p+1} (-1)^i \alpha(g_0,\dots,\hat{g_i},\dots,g_{p+1}).
\end{eqnarray*}
The complex $(C_{c,b}^p(G;A),\delta)$ is called the \textit{bounded homogeneous resolution}.  We define the \textit{continuous bounded cohomology groups} as the cohomology of this complex, i.e.
\[
H_{c,b}^p(G;A):=\frac{\ker(\delta:C_{c,b}^p(G;A)^G\to C_{c,b}^{p+1}(G;A)^G)}{\text{im}(\delta:C_{c,b}^{p-1}(G;A)^G\to C_{c,b}^p(G;A)^G)}.
\]
Forgetting about the boundedness condition, i.e. considering the $G$-modules 
\[
C_{c}^p(G;A):=\{f:G^{p+1}\to A\mid f \text{ continuous}\},
\]
we obtain the \textit{homogeneous resolution}   $(C_{c}^p(G;A),\delta)$ and the \textit{continuous cohomology groups}
\[
H_c^p(G;A):=\frac{\ker(\delta:C_c^p(G;A)^G\to C_c^{p+1}(G;A)^G)}{\text{im}(\delta:C_c^{p-1}(G;A)^G\to C_c^p(G;A)^G)}. 
\]
For more details on continous (bounded) cohomology see \cite{BW}, \cite{Gui} and \cite{Mon1}. 
\\\\
The supremum norm on $C_{c,b}^p(G;A)^G$ induces a seminorm $\|\cdot\|_\infty$ on $H^p_{c,b}(G;A)$ defined by taking the infimum over all supremum norms, i.e. 
\[
\|[\beta]\|_\infty:= \mathrm{inf}_{f\in[\beta]}\mathrm{sup}_{\bar{g}\in G^{p+1}} f(\bar{g}). 
\]
This norm is called the \textit{Gromov norm}. 
\\\\
Let now $G$ be a semisimple Lie group with finite center and no compact factors and let $K$ be its maximal compact subgroup. We denote by $\mathcal{X}$ the associated symmetric space. An important resolution for continuous cohomology is given by the complex of $G$-invariant differential forms $(\Omega^*(\mathcal{X};\mathbb{R})^G,d)$. An isomorphism with the standard homogeneous resolution is provided by the explicit description on the cochain level of the van Est isomorphism by Dupont:

\begin{Th}[van Est isomorphism]
The continuous cohomology of $G$ with real coefficients is isomorphic to $\Omega^*(\mathcal{X};\mathbb{R})^G$. An explicit description on the cocycle level sends the differential form $\omega\in\Omega^p(\mathcal{X};\mathbb{R})^G$ to the cocycle $c_\omega\in C_c(G^{p+1};\mathbb{R})$ defined by
\[
c_\omega(g_0,\dots,g_p)=\int_{\Delta(g_0x,\dots,g_px)}\omega,
\]
for any fixed basepoint $x$ in $\mathcal{X}$. Here $\Delta(g_0x,\dots,g_px)$ is the ``geodesic coned simplex" with vertices $g_0x,\dots,g_px$ defined inductively: The simplex $\Delta(g_0x,g_1x)$ is the geodesic segment from $g_0x$ to $g_1x$ and given $\Delta(g_0x,\dots, g_ix)$ the simplex $\Delta(g_0,\dots, g_ix,g_{i+1}x)$ is the union of all geodesic segments from $g_{i+1}x$ to the points of $\Delta(g_0x,\dots, g_ix)$.
\end{Th}

\subsubsection{The volume form}
Let $\Phi$ be the K\"ahler form on $\Hnc$ and let $(p_1,p_2,p_3)\in(\Hnc\cup\dHnc)^3$. 
Define
\[
c_\Phi(p_1,p_2,p_3) =\int_{\Delta(p_1,p_2,p_3)} \Phi.
\]
Fix a base point $x\in Hnc$. Then for a triple of points in $\Hnc$ the function $c_\Phi$ is equal to the image of $\Phi$ under the van Est isomorphism as defined above evaluated at the point $(g_1,g_2,g_3)$ with $g_i$ such that $x=g_i^{-1}p_i$. For a triple of points in the boundary one can take the limit $x\to\xi$ with $\xi\in\dHnc$ and the cocycle $c_\Phi$ will still represent the same class. From now on we will consider $c_\Phi$ as a map $(\dHnc)^3\to\mathbb{R}$. We clearly have

\begin{Le}
\label{c is cocycle}
$c_\Phi$ is a cocycle, i.e.
\begin{equation}
\label{c cocycle}
 c_\Phi(p_1,p_2,p_3)-c_\Phi(p_1,p_2,p_4)+c_\Phi(p_1,p_3,p_4)-c_\Phi(p_2,p_3,p_4)=0,
 \end{equation}
 for all $(p_1,p_2,p_3,p_4)\in(\dHnc)^4$. 
\end{Le}

The cocycle $c_\Phi$ is proportional to the Cartan angular invariant:
\begin{Th}{\cite[Theorem 7.1.11]{Gol}}
\label{c=2A}
$c_\Phi(\overline{p}) =2\mathbb{A}(\overline{p})$ for all $\overline{p}\in(\dHnc)^3$. 
\end{Th}

\begin{Le}
Let $\omega\in H^4_{c}(\mathrm{PU}(2,1);\mathbb{R)}$ be the image of the volume form in $\Omega^4(\Hc)^{\mathrm{PU}(2,1)}$ under the van Est isomorphism. Then $\omega=\frac{1}{2}\cdot c_\Phi\cup c_\Phi$.
\end{Le}
\begin{proof}
 The volume form on $\Hc$ is equal to $\frac{1}{2}\Phi\wedge\Phi$ (see for example \cite[Chapter 3]{Gol}). Since the van Est isomorphism is natural with respect to products it sends $\frac{1}{2}\Phi\wedge\Phi$ to $\frac{1}{2}c_\Phi\cup c_\Phi$.
\end{proof}

\subsection{Simplicial volume and proportionality principles}
Let $M$ be an oriented closed manifold of dimension $n$. The $\ell^1$-norm $\|\cdot\|_1$ with respect to the basis of singular simplices on the space of real-valued chains $C_*(M;\mathbb{R})$ is given by
\[
\left\|\sum_{j=1}^k a_j\sigma_j\right\|_1=\sum_{j=0}^k |a_j|.
\]
The \textit{$\ell^1$-seminorm} of a homology class in $\alpha\in H_*(M;\mathbb{R})$ is then defined as the infimum of the $\ell^1$-norm of its representatives, i.e.
\[
\|\alpha\|_1:=\inf\left\{\sum_{j=0}^k |a_j| \middle| \alpha=\left[\sum_{j=1}^k a_j\sigma_j\right]\right\}.
\]
The \textit{simplicial volume} $\|M\|$ of $M$ is the $\ell^1$-seminorm of the the real valued fundamental class $[M]\in H_n(M;\mathbb{R})$:
\[
\|M\|:=\|[M]\|_1.
\]

Denote by $\la\beta,\alpha\ra$ the canonical pairing of a cohomology class $\beta\in H^p(M;\mathbb{R})$ with a homology class $\alpha\in H_p(M;\mathbb{R})$. Recall that the \textit{Gromov norm} $\|\beta\|_\infty$ is the semi-norm given by the infimum of the sup-norms of all cocycles representing $\beta$:
\[
\| \beta \|_\infty = \inf \{  \|b\|_\infty \mid [b]=\beta\} \in \mathbb{R}_{\geq 0} \cup \{+ \infty \}.
\]
Then we have:
\begin{Pro}{\cite[Proposition F.2.2]{BP}}
\label{pairing}
For any $\alpha\in H_p(M;\mathbb{R})$ and $\beta\in H^p(M;\mathbb{R})$
\[
|\la\beta,\alpha\ra|\leq\|\beta\|_\infty \cdot \|\alpha\|_1.
\]
Furthermore, if $\beta\in H^n(M;\mathbb{R})$ then
\[
|\la \beta,[M]\ra|=\|\beta\|_\infty\cdot \|M\|.
\]
\end{Pro}
The simplicial volume and the volume  of $M$ are related by the Gromov-Thurston proportionality principal (see \cite{Gr,Th1}) which is given by
\[
\|M\|=\frac{\text{Vol}(M)}{c(\widetilde{M})},
\]
where $c(\widetilde{M})$ is a positive constant (possibly infinite) which only depends on the universal cover $\widetilde{M}$ of $M$. In fact, for locally symmetric spaces of noncompact type, Bucher obtains in \cite{Bu1} that the proportionality constant
\[
c(\widetilde{M})=\|\omega_{\widetilde{M}}\|_\infty,
\]
with $\omega_{\widetilde{M}}\in H_c^n(\text{Isom}_0(\widetilde{M});\mathbb{R})$ the volume form. So we have
\begin{Pro}
\label{proportionality}
Let $M$ be a locally symmetric space of noncompact type. Then 
\[
\|M\|=\frac{\mathrm{Vol}(M)}{\|\omega_{\widetilde{M}}\|_\infty},
\]
with $\omega_{\widetilde{M}}\in H_c^n(\text{Isom}_0(\widetilde{M});\mathbb{R})$ the volume form.
\end{Pro}

The volume of $M$ is also proportional to $\chi(M)$. Indeed, by Hirzebruch's proportionality principle,
\begin{Pro}{\cite{Hi}}
\label{Hirzebruch}
Let $N$ be a closed, oriented, locally symmetric space of noncompact type and let $\mathcal{X}_u$ be the compact dual of its universal cover $\mathcal{X}=\widetilde{N}$. Then
\[
\frac{\mathrm{Vol}(N)}{\chi(N)}=\frac{\mathrm{Vol}(\mathcal{X}_u)}{\chi(\mathcal{X}_u)}.
\] 
\end{Pro}

\section{Results}
\label{Results}

Let $M$ be a closed manifold that is locally isometric to the complex hyperbolic plane $\mathbb{H}_{\mathbb{C}}^2$.

Combining Proposition \ref{proportionality} with Theorem \ref{kahler} we obtain
\begin{namedthm*}{Theorem \ref{simplicialvolume}}
Let $M$ be a closed oriented manifold which is locally isometric to $\Hc$. Then
\[
\frac{2}{\pi^2}\mathrm{Vol}(M)\leq \|M\|\leq \frac{9}{4\pi^2}\mathrm{Vol}(M).
\]
\end{namedthm*}
We can also express this in terms of $\chi(M)$. Using Hirzeburch's proportionality principle, i.e. Proposition \ref{Hirzebruch}, we get 

\begin{Le}
\[
\mathrm{Vol}(M)=\frac{8}{3}\pi^2\cdot \chi(M).
\]
\end{Le}
\begin{proof}
The compact dual of $\Hc$ is the complex projective plane $\mathbb{C}P^2$. Therefore, by Proposition \ref{Hirzebruch},
\[
\frac{\mathrm{Vol}(M)}{\chi(M)}=\frac{\mathrm{Vol}(\mathbb{C}P^2)}{\chi(\mathbb{C}P^2)}. 
\]
As $\mathbb{C}P^2$ has zero homology groups in odd dimensions and one-dimensional homology groups in even dimensions $\chi(\mathbb{C}P^2)=3$. Furthermore, the complex projective plane $\mathbb{C}P^2$ is a symplectic quotient $S^{5}(r)/S^1(r)$, with $S^n(r)$ the real $n$-sphere of radius $r$. Hence the volume of $\mathbb{C}P^2$ is
\[
\mathrm{Vol}(\mathbb{C}P^2)=\mathrm{Vol}(S^5(r))/\mathrm{Vol}(S^1(r))=\pi^3r^5/2\pi r=\frac{1}{2}\pi^2r^4,
\]
To have holomorphic sectional curvature equal to $1$ and therefore sectional curvature between $1/4$ and $1$ we have to set $r=2$ and we therefore obtain $\mathrm{Vol}(\mathbb{C}P^2)=8\pi^2$. For a more elaborate description of the Fubini-Study metric on the complex projective space and its sectional curvature see e.g. the  first pages of Chapter 6 in \cite{Sak}.
\end{proof}

We get the following lower and upper bound for the simplicial volume of $M$ in terms of its Euler characteristic:
\begin{namedthm*}{Corollary \ref{ineq_chi}}
Let $M$ be a closed oriented manifold which is locally isometric to $\Hc$. Then
\[
\frac{16}{3}\chi(M)\leq \|M\|\leq 24\chi(M). 
\]
\end{namedthm*}

Recall that a $\mathrm{GL}^+(4,\mathbb{R})$-bundle over $M$ is called \textit{flat} if it admits a flat structure, i.e. a connection on $\xi$ with zero curvature. Let $\xi$ be any flat $\mathrm{GL}^+(4,\mathbb{R})$-bundle over $M$. Its Euler number $\chi(\xi)$ is by definition the pairing of the Euler class $\varepsilon_4(\xi)\in H^4(M;\mathbb{R})$ with the fundamental class $[M]\in H_4(M;\mathbb{R})$:
\[
\chi(\xi):=<\varepsilon_4(\xi),[M]>.
\]
By Proposition \ref{pairing}
\[
|\chi(\xi)|=\|\varepsilon_4(\xi)\|_\infty\cdot\| M\|.
\]
Furthermore, Ivanov and Turaev \cite{IT} showed that the sup norm of the Euler class $\varepsilon_4(\xi)$ satisfies
\[
\|\varepsilon_n(\xi)\|_\infty\leq \frac{1}{2^n}.
\]

Thus from Corollary \ref{ineq_chi} we obtain the Milnor-Wood inequality 

\begin{namedthm*}{Corollary \ref{MW}}
Let $\xi$ be a flat $\mathrm{GL}^+(4,\mathbb{R})$-bundle over a closed complex hyperbolic surface $M$. Then
\[
|\chi(\xi)|\leq \frac{3}{2}\chi(M).
\]
\end{namedthm*}

\section{Proof of Theorem \ref{kahler}}

By Proposition \ref{proportionality}
\[
\|M\|= \frac{\text{Vol}(M)}{\|\omega\|_\infty},
\]
where $\omega=\frac{1}{2}c_\phi\cup c_\phi\in H_c^4(\mathrm{PU}(2,1);\mathbb{R})$ is the image under the van Est isomorphism of the volume form.

\begin{proof}[Proof of Theorem \ref{kahler}]
For all triples $(x_0,x_1,x_2)\in(\dHc)^3$ we have 
\[
|c_\Phi(x_0,x_1,x_2)|\leq \pi,
\]
and we therefore obtain the trivial upper bound $\pi^2$ for $\| [c_\Phi\cup c_\Phi]\|_\infty$. In Proposition \ref{lowerbound} below we obtain the lower bound $\frac{2}{9}\pi^2$. 
\end{proof}

Let $X$ be a topological space. The alternation of a $p$-cochain $f:X^{p+1}\to\mathbb{R}$ is given by 
\[
\text{Alt}(f)(x_0,\dots,x_p)=\frac{1}{(p+1)!}\sum_{\sigma\in \text{Sym}(p+1)}\text{sign}(\sigma)f(x_{\sigma(0)},\dots,x_{\sigma(p)}).
\]
Recall that for a $p$-cochain $f:X^{p+1}\to\mathbb{R}$ and a $q$-cochain $g:X^{q+1}\to\mathbb{R}$ the standard cup product $f\cup g$ is the $p+q$-cochain defined by 
\[
f\cup g(x_0,\dots, x_{p+q})=f(x_0,\dots,x_p)g(x_p,\dots,x_{p+q}).
\]
Slightly abusing notation, we will still denote by $c_\Phi\cup c_\Phi$ the alternation of the standard cup product of $c_\Phi$ with itself. Thus $c_\Phi\cup c_\Phi(x_0,\dots,x_4)$ is equal to
\begin{align*}
\frac{1}{120}\sum_{\sigma\in\Sym}\mathrm{sign}(\sigma)c_\Phi(x_{\sigma(0)},x_{\sigma(1)},x_{\sigma(2)})\cdot c_\Phi(x_{\sigma(2)},x_{\sigma(3)},x_{\sigma(4)}).
\end{align*}
Any $\sigma\in\mathrm{Sym}(5)$ can be uniquely written as $\tau^k\circ\alpha$, where $\alpha\in\mathrm{Sym}(5)$ maps $2$ to $0$, $\tau=(0\,1\,2\,3\,4)$ and $k$ is an integer from $0$ to $4$. Then, exploiting the fact that $c_\Phi$ itself is already alternating, we get
\begin{align*}
c_\Phi\cup c_\Phi&(x_0,\dots,x_4)\\
&=\frac{1}{15}\sum_{\substack{\tau=(0\,1\dots4)^k\\ k\in\{0,1,\dots,4\}}} c_\Phi(x_{\tau(0)},x_{\tau(1)},x_{\tau(2)})\cdot c_\Phi(x_{\tau(0)},x_{\tau(3)},x_{\tau(4)}) \\
&-c_\Phi(x_{\tau(0)},x_{\tau(1)},x_{\tau(3)})\cdot c_\Phi(x_{\tau(0)},x_{\tau(2)},x_{\tau(4)})\\
&+c_\Phi(x_{\tau(0)},x_{\tau(1)},x_{\tau(4)})\cdot c_\Phi(x_{\tau(0)},x_{\tau(2)},x_{\tau(3)}).
\end{align*}
We use the cocycle relation 
\begin{align*}
0&=\delta c_\Phi(x_0,x_i,x_j,x_k)\\
&=c_\Phi(x_i,x_j,x_k)-c_\Phi(x_0,x_j,x_k)\\
&+ c_\Phi(x_0,x_i,x_k)-c_\Phi(x_0,x_i,x_j),
\end{align*}
to rewrite the above sum such that $c_\Phi$ is always evaluated at $x_0$ and two other points. It then turns out that the five terms in the above sum are all the same and we therefore obtain that $c_\Phi\cup c_\Phi(x_0,\dots,x_4)$ is equal to 
\begin{align}
\label{alternatingcup}
\frac{1}{3}\big[&c_\Phi(x_0,x_1,x_2)\cdot c_\Phi(x_0,x_3,x_4)-c_\Phi(x_0,x_1,x_3)\cdot c_\Phi(x_0,x_2,x_4) \nonumber \\
&+c_\Phi(x_0,x_1,x_4)\cdot c_\Phi(x_0,x_2,x_3)\big].
\end{align}
\begin{Rem}
The natural map $Alt^*:H_c^p(X;\mathbb{R})\to H_c^p(X;\mathbb{R})$ which sends a cocycle to its alternation is an isomorphism with inverse induced by the identity map. Since both maps do not increase norms at the cochain level, we have $\|[Alt(f)]\|_\infty=\|[f]\|_\infty$ for any $[f]\in H_c^p(X;\mathbb{R})$. 
\end{Rem}

\subsection{Lower bound}
Our strategy for finding a lower bound is to find a set of $5$-tuples $\overline{\mathbf{p}_i}\in (\dHc)^5$ such that for all $\mathrm{PU}(2,1)$-invariant alternating cochains $b:(\dHc)^4\to\mathbb{R}$ we have
\[
\sum \delta b(\overline{\mathbf{p}_i})=0.
\]
If $b\in L^\infty((\dHc)^4;\mathbb{R})$ then $b$ is only defined a.e. and thus a pointwise equality as above would have no meaning. However, from Lemma \ref{discrete} and Lemma \ref{normlattice} it will follow that we can instead consider $c_\Phi\cup c_\Phi:(\dHc)^5\to\mathbb{R}$ as a cocycle in $H_b^4(\mathrm{PU}(2,1)^\delta;\mathbb{R})$, with $\mathrm{PU}(2,1)^\delta$ the underlying discrete group of the topological group $\mathrm{PU}(2,1)$. We will denote this everywhere defined cocycle by $c_\Phi^\delta\cup c_\Phi^\delta$. 

\begin{Le}
\label{discrete}
The bounded cohomology group $H_b^*(\mathrm{PU}(2,1)^\delta;\mathbb{R})$ is realized on the boundary, i.e. by the resolution 
\[
0\to \ell^\infty(\dHc;\mathbb{R})^G\to \ell^\infty((\dHc)^2;\mathbb{R})^G\to \ell^\infty((\dHc)^3;\mathbb{R})^G\to\dots
\]
\end{Le}
\begin{proof}
The minimal parabolic subgroup of $\mathrm{PU}(2,1)$ is the Heisenberg similarity group $\mathbf{Sim}(\mathcal{H})=(\mathbb{R}_+\times U(1))\rtimes \mathfrak{N}$. This group is amenable as an abstract group and thus, by \cite[Theorem 7.5.3]{Mon1},  $H_b(\mathrm{PU}(2,1)^\delta;\mathbb{R})$ is given by the cohomology of the complex $(\ell^\infty((\mathrm{PU}(2,1)/ \mathbf{Sim}(\mathcal{H}))^{*+1};\mathbb{R})^{\mathrm{PU}(2,1)},\delta)$. 
\end{proof}

\begin{Le}
\label{normlattice}
Let $\Gamma=\mathrm{PU}(2,1;\mathbb{Z}[i])$. Then $\|[c_\Phi\cup c_\Phi]\|_\infty$ is equal to 
\[
\mathrm{inf}\{\|c_\Phi^\delta\cup c_\Phi^\delta+\delta b\|_{\ell^\infty} \mid b:(\dHc)^4\to\mathbb{R} \mathrm{\,bounded\, and\,}\Gamma\mathrm{-invariant}\}
\]
\end{Le}
\begin{proof}
We follow Section 6 of \cite{BuMo}. Note that for the proof presented there amenability of the minimal parabolic as an abstract group is not necessary. Let $G=\mathrm{PU}(2,1)$ and let $P$ be its minimal parabolic subgroup. Denote by $\mathscr{L}^\infty((G/P)^{p+1};\mathbb{R})$ the Banach $G$-module of bounded measurable functions $(G/P)^{p+1}\to\mathbb{R}$ so that we have the natural quotient map 
\[
q:\mathscr{L}^\infty((G/P)^{p+1};\mathbb{R})\twoheadrightarrow L^\infty((G/P)^{p+1};\mathbb{R}), 
\]
and the natural inclusion map
\[
i:\mathscr{L}^\infty((G/P)^{p+1};\mathbb{R})\hookrightarrow \ell^\infty((G/P)^{p+1};\mathbb{R}).
\]
Then $q(c_\Phi\cup c_\Phi)$ is the function class of $c_\Phi\cup c_\Phi$ in  $L^\infty((G/P)^{5};\mathbb{R})^G$ (which by slight abuse of notation we also denote by $c_\Phi\cup c_\Phi$ in the rest of this text). On the other hand, $i(c_\Phi\cup c_\Phi)=c_\Phi^\delta\cup c_\Phi^\delta\in\ell^\infty((G/P)^{p+1};\mathbb{R})^G$. The restriction maps $\mathrm{res}_{(c),b}:H_{(c,)b}^4(G;\mathbb{R})\to H_b^4(\Gamma;\mathbb{R})$ send $[q(c_\Phi\cup c_\Phi)]$ and $[i(c_\Phi\cup c_\Phi)]$ to the same cohomology class in $H_b^4(\Gamma;\mathbb{R})$. It follows that 
 \[
 \|[\mathrm{res}_{c,b}(q(c_\Phi\cup c_\Phi))]\|_{\ell^\infty}=\|[\mathrm{res_b}(i(c_\Phi \cup c_\Phi))]\|_{\ell^\infty}. 
 \]
Since restricting to a cocompact lattice preserves the seminorm in continuous bounded comology \cite[Proposition 8.6.2]{Mon1} we can conclude 
\[
 \|[q(c_\Phi\cup c_\Phi)]\|_{\infty}=\|[\mathrm{res_b}(i(c_\Phi \cup c_\Phi))]\|_{\ell^\infty}.
\]
Furthermore, note that the restriction map $\mathrm{res}_b$ is realized by the inclusion $\ell^\infty((G/P)^{p+1};\mathbb{R})^G\hookrightarrow \ell^\infty((G/P)^{p+1};\mathbb{R})^\Gamma$. This finishes the proof. 
\end{proof} 

Let
\[
x_+=\begin{bmatrix} 1\\ 0\\1\end{bmatrix}, x_i=\begin{bmatrix} i\\0\\1\end{bmatrix}, y_+=\begin{bmatrix} 0\\1\\1\end{bmatrix}, y_i=\begin{bmatrix} 0\\i\\1\end{bmatrix}, y_{-i}=\begin{bmatrix}0\\-i\\1\end{bmatrix}, v=\begin{bmatrix} \frac{1}{2}(1+i) \\ \frac{1}{2}(1+i) \\ 1\end{bmatrix},
\]
be points in the boundary of the complex hyperbolic plane in the projective model. We have
\begin{align*}
&\A(x_+,x_i,y_+)=\A(x_+,x_i,y_i)=\A(x_+,x_i,y_{-i})=\A(x_+,y_+,y_i)=\frac{\pi}{4},\\
&\A(x_+,y_+,y_{-i})=\A(x_+,x_i,v)=-\frac{\pi}{4},\\
&\A(x_+,y_i,y_{-i})=\A(x_+,y_+,v)=0, \\
&\A(x_+,y_i,v)=-\frac{\pi}{2}.
\end{align*}
Recall that equation \ref{alternatingcup} gives a convenient way for calculating the alternating cup product $c_\Phi\cup c_\Phi$ and furthermore that $c_\Phi=2\A$. Therefore 
\begin{align*}
c_\Phi^\delta&\cup c_\Phi^\delta(x_+,x_i,y_+,y_i,y_{-i})\\
&=\frac{1}{3}\left[c_\Phi^\delta(x_+,x_i,y_+)c_\Phi^\delta(x_+,y_i,y_{-i})-c_\Phi^\delta(x_+,x_i,y_i)c_\Phi^\delta(x_+,y_+,y_{-i})\right.\\
&\hspace{.5cm}\left. +c_\Phi^\delta(x_+,x_i,y_{-i})c_\Phi^\delta(x_+,y_+,y_i)\right]\\
&=\frac{1}{3}\cdot\left[\frac{\pi}{2}\cdot 0-\frac{\pi}{2}\cdot\left(-\frac{\pi}{2}\right)+\frac{\pi}{2}\cdot\frac{\pi}{2}\right]\\
&=\frac{1}{6}\pi^2, 
\end{align*}
and
\begin{align*}
c_\Phi^\delta&\cup c_\Phi^\delta(x_+,x_i,y_+,y_i,v)\\
&=\frac{1}{3}\left[c_\Phi^\delta(x_+,x_i,y_+)c_\Phi^\delta(x_+,y_i,v)-c_\Phi^\delta(x_+,x_i,y_i)c_\Phi^\delta(x_+,y_+,v)\right.\\
&\hspace{.5cm} \left.+c_\Phi^\delta(x_+,x_i,v)\cdot c_\Phi^\delta(x_+,y_+,y_i)\right]\\
&= \frac{1}{3}\cdot\left[\frac{\pi}{2}\cdot(-\pi)-\frac{\pi}{2}\cdot 0+\left(-\frac{\pi}{2}\right)\cdot\frac{\pi}{2}\right]\\
&=-\frac{1}{4}\pi^2.
\end{align*}

\begin{Le}
\label{oct}
Let $b:(\dHc)^4\to\mathbb{R}$ be an alternating $\Gamma$-invariant cochain. Then $\delta b(x_+,x_i,y_+,y_i,y_{-i}) = 2b(x_+,x_i,y_+,y_i)$.
\end{Le}
\begin{proof}
By definition
\begin{align*}
\delta b(x_+,x_i,y_+,y_i,y_{-i})= b&(x_i,y_+,y_i,y_{-i})-b(x_+,y_+,y_i,y_{-i})\\ 
&+b(x_+,x_i,y_i,y_{-i}) -b(x_+,x_i,y_+,y_{-i})\\
& +b(x_+,x_i,y_+,y_i).
\end{align*}
Denote by $L_x$ the complex line that contains $x_+$ and $x_i$ and by $L_y$ the complex line that contains $y_+,y_i$ and $y_{-i}$. The reflection in $L_x$, represented by the matrix 
\[
\begin{bmatrix} 1 & 0 & 0 \\ 0 & -1 & 0 \\ 0 & 0 &1 \end{bmatrix},
\]
exchanges $y_i$ and $y_{-i}$ while fixing $x_+$ and $x_i$. Thus, as $b$ is alternating,
\begin{equation}
\label{eq1}
b(x_+,x_i,y_i,y_{-i})=0.
\end{equation}
The reflection in $L_y$ with reflection factor $-i$, represented by the matrix
\[
\begin{bmatrix} -i & 0 & 0 \\ 0 &1 & 0 \\ 0 & 0 & 1 \end{bmatrix},
\]
sends $x_i$ to $x_+$ while fixing $y_+, y_i$ and $y_{-i}$. It follows that 
\begin{equation}
\label{eq2}
b(x_i,y_+,y_i,y_{-i})=b(x_+,y_+,y_i,y_{-i}).
\end{equation}
Lastly, the reflection in $L_x$ with reflection factor $i$, represented by the matrix
\[
\begin{bmatrix} 1 & 0 & 0 \\ 0 & i & 0 \\ 0 & 0 & 1 \end{bmatrix},
\]
maps $y_+\mapsto y_i$ and $y_{-i}\mapsto y_+$ while fixing $x_+$ and $x_i$. It follows that $b(x_+,x_i,y_+,y_{-i})=b(x_+,x_i,y_i,y_+)$ and hence, since $b$ is alternating,
\begin{equation}
\label{eq3}
b(x_+,x_i,y_+,y_{-i})=-b(x_+,x_i,y_+,y_i).
\end{equation}
Combining equations \ref{eq1},\ref{eq2} and \ref{eq3} gives 
\[
\delta b(x_+,x_i,y_+,y_i,y_{-i}) = 2b(x_+,x_i,y_+,y_i).
\] 
\end{proof}

\begin{Le}
\label{cube}
Let $b:(\dHc)^4\to\mathbb{R}$ be an alternating $\Gamma$-invariant cochain. Then $\delta b(x_+,x_i,y_+,y_i,v)=b(x_+,x_i,y_+,y_i)$.
\end{Le}
\begin{proof}
The isomorphism represented by the matrix
\[
\begin{bmatrix} 0 & 1 & 0 \\ 1 & 0 & 0 \\ 0 & 0 & 1\end{bmatrix} 
\]
exchanges $x_+$ with $y_+$, and $x_i$ with $y_i$ while fixing $v$. Combined with the fact that $b$ is alternating this gives
\begin{align*}
&b(x_i,y_+,y_i,v)  =  b(x_+,x_i,y_i,v),\, \mathrm{and} \,\,\, b(x_+,y_+,y_i,v) =  b(x_+,x_i,y_+,v). 
\end{align*}
Thus
\[
\delta b(x_+,x_i,y_+,y_i,v)=2b(x_+,x_i,y_i,v)-2b(x_+,x_i,y_+,v)+b(x_+,x_i,y_+,y_i).
\]
Furthermore, the isomorphism represented by the matrix
\[
\begin{bmatrix} -1+i & 0 & 1 \\ 0 & -i & 0 \\ i & 0 & 1-i \end{bmatrix}
\]
sends the $4$-tuple $(x_+,x_i,y_i,v)$ to $(x_i,x_+,v,y_+)$ and thus
\[
b(x_+,x_i,y_i,v)=b(x_+,x_i,y_+,v).
\]
It follows that $\delta b(x_+,x_i,y_+,y_i,v)=b(x_+,x_i,y_+,y_i)$
\end{proof}

 \begin{Rem}
Note that
\[
\A(x_+,x_i,y_+)=\A(x_+,x_i,y_i)=\A(x_+,y_+,y_i)=\A(x_i,y_+,y_i)=\frac{\pi}{4},
\]
and thus the $4$-tuple $(x_+,x_i,y_+,y_i)$ is a  regular special symmetric tetrahedron as defined in \cite{Fal}. Coordinates in the Heisenberg model for such a tetrahedron (with $\A=\pi/4$) are for example $\infty, 0, (1,1)$ and $(i,1)$.
\end{Rem}

 \begin{Rem}
 The $8$ vectors
 \[
 \begin{bmatrix} \pm 1\\0\\1\end{bmatrix}, \begin{bmatrix} \pm i\\0\\1\end{bmatrix}, \begin{bmatrix} 0\\ \pm 1\\ 1\end{bmatrix}, \begin{bmatrix} 0\\ \pm i\\1\end{bmatrix},
 \]
 correspond to the eight vertices of a regular octahedron in $\mathbb{R}^4$ with edge length $\sqrt{2}$. The $5$-tuple $(x_+,x_i,y_+,y_i,y_{-i})$ corresponds to one of the simplices in the minimal triangulation of this octahedron. In fact, $c_\Phi^\delta\cup c_\Phi^\delta$ takes the value $\pm \pi^2/6$ on all the simplices of this triangulation. Furthermore the eight vertices of the form
 \[
 \begin{bmatrix} \frac{1}{2}(\pm 1\pm i)\\ \frac{1}{2}(\pm 1\pm i)\\1\end{bmatrix},
 \]
 with an even number of plus signs also correspond to a regular octahedron in $\mathbb{R}^4$ with edge length $\sqrt{2}$. Together the $16$ vertices correspond to a regular cube in $\mathbb{R}^4$ with edge length $1$. The $5$-tuple $(x_+,x_i,y_+,y_i,v)$ is one of the simplices in the minimal triangulation of this cube found in \cite{Mara}. It corresponds to one of the eight corners that are ``sliced off" in this construction and in fact $c_\Phi^\delta\cup c_\Phi^\delta$ is equal to $\pm \pi^2/4$ on all these eight simplices while on the remaining eight simplices in the triangulation it is again equal to $\pm \pi^2/6$. 
 \end{Rem}

\begin{Pro}
\label{lowerbound}
$\|[c_\Phi\cup c_\Phi]\|_\infty\geq \frac{2}{9}\pi^2$.
\end{Pro}
\begin{proof}
Note that since $c_\Phi^\delta\cup c_\Phi^\delta$ is alternating we can restrict to alternating cochains to compute  $\mathrm{inf}\|c_\Phi^\delta\cup c_\Phi^\delta +\delta b\|_{\ell^\infty}$. By Lemma \ref{oct} and Lemma \ref{cube} 
\[
\delta b(x_+,x_i,y_+,y_i,y_{-i}) - 2\delta b (x_+,x_i,y_+,y_i,v)=0,
\]
for all alternating cochains $b\in\ell^\infty((\dHc)^4;\mathbb{R})^\Gamma$. Let $\mathbf{p}_1=(x_+,x_i,y_+,y_i,y_{-i})$ and $\mathbf{p}_2= (x_+,x_i,y_+,y_i,v)$. Then 
\begin{align*}
 \|c_\Phi^\delta\cup c_\Phi^\delta+\delta b\|_{\ell^\infty}&\geq \frac{1}{3}\left((c_\Phi^\delta\cup c_\Phi^\delta+\delta b)(\mathbf{p}_1)-2(c_\Phi^\delta\cup c_\Phi^\delta+\delta b)(\mathbf{p}_2)\right)\\
 &=\frac{1}{3}\left( c_\Phi^\delta\cup c_\Phi^\delta(\mathbf{p}_1)-2c_\Phi^\delta\cup c_\Phi^\delta(\mathbf{p}_2)\right)\\
 &=\frac{2}{9}\pi^2,
\end{align*}
for all alternating cochains $b\in \ell^\infty((\dHc)^4;\mathbb{R})^\Gamma$ and it therefore follows from Lemma \ref{normlattice} that 
\[
\|[c_\Phi\cup c_\Phi]\|_\infty\geq \frac{2}{9}\pi^2.
\]
\end{proof}

\begin{Rem}
Let $\Lambda=\mathrm{PU}(2,1;\mathbb{Z}[\omega])$ be the Eisenstein-Picard modular group which is by definition the subgroup of $\mathrm{PU}(2,1)$ with entries in the ring $\mathbb{Z}[\omega]$ where $\omega$ is a cube root of unity. Let $\Lambda_\infty$ be the stabilizer of $\infty$ in $\Lambda$ and let $\Lambda_T<\Lambda_\infty$ be its torsion-free subgroup. A $5$-tuple that realizes the lower bound $2\pi^2/9$ is given by the following points in the Heisenberg space:
\[
(0,-\sqrt{3}),\, (-\omega,0),\, (1,0), \, (0,\sqrt{3}),\, (0,2\sqrt{3}).
\]
These are the vertices $p_0,p_1,p_2,p_3$ and $p_8$ of the fundamental domain of $\Lambda_T$ described in \cite[Annexe A]{Ge}. 
\end{Rem}


\end{document}